\documentclass[11pt]{amsart}
\newcommand{\ben}{\begin{enumerate}}
\newcommand{\een}{\end{enumerate}}
\newcommand{\eq}[2][label]{\begin{equation}\label{#1}#2\end{equation}}
\newcommand{\av}[2]{\langle #1\rangle_{_{\scriptstyle #2}}}

\newcommand{\ve}{\varepsilon}
\usepackage{amsmath}
\usepackage{amsfonts}
\usepackage{amssymb}
\usepackage{amsthm}
\usepackage{graphicx,color,hyperref}
\usepackage{mathtools}

\usepackage{enumitem}

\newcommand{\bel}[1]{\boldsymbol{#1}}

\newcommand{\BMO}{{\rm BMO}}

\newcommand{\rn}{\mathbb{R}^n}

\newtheorem{theorem}{Theorem}[section]

\newtheorem{lemma}[theorem]{Lemma}
\newtheorem{corollary}[theorem]{Corollary}
\newtheorem{prop}[theorem]{Proposition}

\newtheorem*{theorem*}{Theorem}{\bf}{\it}
\newtheorem*{proposition*}{Proposition}{\bf}{\it}
\newtheorem*{observation*}{Observation}{\bf}{\it}
\newtheorem*{lemma*}{Lemma}{\bf}{\it}
\newtheorem*{conjecture*}{Conjecture}{\bf}{\it}

\theoremstyle{definition}

\theoremstyle{remark}
\newtheorem{remark}[theorem]{Remark}

\numberwithin{equation}{section}

\allowdisplaybreaks
\begin{document}

\title[Dimension-free estimates]{Dimension-free estimates for semigroup BMO and $A_p$}

\author{Leonid Slavin}
\author{Pavel Zatitskii}
\address{University of Cincinnati, P.O. Box 210025, OH 45221-0025, USA}

\email{leonid.slavin@uc.edu}

\address{St. Petersburg State University, 29B 14th Line V.O., Saint Petersburg, 199178, Russia \and St. Petersburg Department of Steklov Mathematical Institute of Russian Academy of Sciences,
27~Fontanka, Saint Petersburg, 191023, Russia}

\email{pavelz@pdmi.ras.ru}

\thanks{The first author is supported by the Simons Foundation, collaboration grant 317925}
\thanks{The second author is supported by the Russian Science Foundation: Theorem~\ref{main_bmo}, Corollary~\ref{cor_bmo}, and Lemma~\ref{ito} were obtained with support from the RSF grant 19-71-30002}




\begin{abstract}
 Let $K_t$ be either the heat or the Poisson kernel on $\mathbb{R}^n.$ Let $\mathcal{A}$ stand either for BMO equipped with the quadratic seminorm or for $A_p,$ $1< p\le\infty.$ We establish the following transference between the class $\mathcal{A}$ on an interval $I\subset\mathbb{R}$ and its $K$-version, $\mathcal{A}^K,$ on $\rn$: If a given integral functional admits an estimate on $\mathcal{A}(I),$ then the same estimate holds for $\mathcal{A}^K(\rn),$ with all Lebesgue averages replaced by $K$-averages. In particular, all such estimates are dimension-free. 
As an application, via the heat kernel, we obtain a weakly-dimensional theory for $\BMO(\rn)$ on balls. In particular, we show that the John--Nirenberg constant of this space decays with dimension no faster than $n^{-1/2}.$
\end{abstract}
\maketitle

\section{Preliminaries}

For a ball $B\subset \rn$ and $\varphi\in L^1_{loc}(\rn),$ write
$
\av{\varphi}B=\frac1{|B|}\int_B\varphi.
$
The space
$\BMO(\mathbb{R}^{n})$ is the set of all locally integrable, real-valued functions $\varphi$ on $\mathbb{R}^n$ for which
$$
\|\varphi\|^{}_{*}:=\sup_{\text{ball}~B}\av{|\varphi-\av{\varphi}B|^{2}}B^{\frac12}=\sup_{\text{ball}~B}(\av{\varphi^2}B-\av{\varphi}B^2\big)^{\frac12}<\infty.
$$
If $n=1$ and all balls $B$ are subintervals of a given interval $I,$ we write $\BMO(I)$ instead of $\BMO(\mathbb{R})$ and $\|\varphi\|^{}_{*,I}$ instead of $\|\varphi\|^{}_{*}.$

For $1<p<\infty,$ the class $A_p(\rn)$ is the set of all locally integrable, almost everywhere positive functions $w$ (called weights) such that 
$$
[w]_{p}:=\sup_{\text{ball}~B}\av{w}B\av{w^{-\frac1{p-1}}}B^{p-1}<\infty.
$$
The class $A_\infty(\rn)$ is the set of all weights $w$ such that
$$
[w]_\infty:=\sup_{\text{ball}~B}\av{w}Be^{-\av{\log w}{\scriptscriptstyle B}}<\infty.
$$
For all $p,$ the quantity $[w]_p$ is referred to as the $A_p$-characteristic of the weight $w.$ 
As in the case of BMO, we will write $A_p(I)$ and $[w]_{p,I}$ when $n=1$ and all balls $B$ are subintervals of a given finite interval $I.$

For $y\in\rn$ and $t>0,$ the Poisson kernel and the heat kernel are given, respectively, by:
\eq[kernels]{
P_t(y)=\frac{\Gamma\big(\frac{n+1}2\big)}{\pi^{\frac{n+1}2}}\, \frac t{(t^2+|y|^2)^{\frac{n+1}2} },\qquad
H_t(y)=\frac1{(4\pi t)^{\frac n2}}\, e^{-\frac{|y|^2}{4t}}.
}
We will write $K_t,$ or simply $K,$ for both $P_t$ and $H_t$ in statements that apply to both kernels. 

For any sufficiently integrable function $g$ on $\mathbb{R}^n,$ $y\in\mathbb{R}^n,$ and $t>0,$ let
$$
g_{\scriptscriptstyle K}(y,t)=(K_t*g)(y)
$$
be the $K$-extension of $g$ into $\mathbb{R}_+^{n+1}.$ (This convolution gives an operator semigroup.) We will often write $z=(y,t)$ and use the shorthand $g(z)$ for $g_{\scriptscriptstyle K}(z)$ when $K$ can be taken to be either $P$ or $H$ or when it is clear which one is meant.
\medskip

If $\varphi\in\BMO(\rn),$ then $\varphi(z)$ and $\varphi^2(z):=(\varphi^2)_{\scriptscriptstyle K}(z)$ are defined for all $z\in\mathbb{R}_+^{n+1}.$ In fact, it is well known that each of the following is an equivalent norm on $\BMO(\rn)$:
$$
\|\varphi\|_{\scriptscriptstyle P}:=\sup_{z\in\mathbb{R}_+^{n+1}}\big(\varphi^2_{\scriptscriptstyle P}(z)-\varphi_{_P}(z)^2\big)^{\frac 12}, \quad
\|\varphi\|_{\scriptscriptstyle H}:=\sup_{z\in\mathbb{R}_+^{n+1}}\big(\varphi^2_{\scriptscriptstyle H}(z)-\varphi_{_H}(z)^2\big)^{\frac 12}.
$$
(One of the constants of equivalence will come into play in Section~\ref{appl}.) 
To emphasize the choice of the norm, we will sometimes refer to $\BMO$ as $\BMO_*$ or $\BMO_{\scriptscriptstyle K},$ as appropriate.

By analogy with BMO, we define the class $A_p^{\scriptscriptstyle K}(\rn),$ $1<p<\infty,$ to be the set of all weights $w$ for which the following inequality holds:
$$
[w]_{p}^{\scriptscriptstyle K}:=\sup_{z\in\mathbb{R}_+^{n+1}}w(z)\Big[w^{-\frac1{p-1}}(z)\Big]^{p-1}<\infty,
$$
while the class $A_\infty^{\scriptscriptstyle K}(\rn)$ is the set of all weights $w$ such that
$$
[w]_{\infty}^{\scriptscriptstyle K}:=\sup_{z\in\mathbb{R}_+^{n+1}} w(z)e^{-\log w(z)}<\infty.
$$
It is easy to show that $A_p^{\scriptscriptstyle H}(\rn)=A_p(\rn)$ for all $n\ge1$ and all $p,$ and that the corresponding characteristics are equivalent, in the sense of two-sided inequalities. However, if $p<n+1,$ then $A_p^{\scriptscriptstyle P}(\rn)\subsetneq A_p(\rn).$ (To wit, $w_a(x):=|x|^{-an}\in A_p(\rn)$ for $a\in[0,1),$ but $w_a\notin A_p^{\scriptscriptstyle P}(\rn)$ unless $p\ge n+1.)$

Our main result is a transference between integral estimates on $\BMO(I)$ and $\BMO_{\scriptscriptstyle K}(\rn),$ and, separately, between estimates on $A_p(I)$ and $A_p^{\scriptscriptstyle K}(\rn).$ To elaborate,
if $f$ is a non-negative function on $\mathbb{R}$ such that the integral functional $\av{f\circ\eta}I$ is bounded on $\BMO_*(I),$ then exactly the same bound holds for $(f\circ\varphi)(z)$ on $\BMO_{\scriptscriptstyle K}(\rn)$ with all averages over $I$ replaced with $K$-averages, and similarly for $A_p.$ This result follows from a subordination relationship between the corresponding Bellman functions. We do not actually compute any Bellman functions in this paper; instead, they are defined in the abstract, as solutions of extremal problems.
The transference we prove is intuitive -- indeed, Bellman functions in model settings have long been used to obtain estimates in related problems. As is often the case, it amounts to a Jensen-type inequality for the model function. Such inequalities are straightforward when one has a convex or concave function defined on a convex domain. In our situation, however, there are two distinct challenges: the non-convexity of the Bellman domains for BMO and $A_p$ and the apparent lack of information about the structure of abstract Bellman functions. The former is handled with a probabilistic argument dependent on the semigroup nature of the kernel $K.$ The latter is resolved by a recent result from~\cite{sz} which establishes {\it a priori} local concavity of Bellman functions for general averaging classes on an interval, including $\BMO$ and $A_p.$

Our main application is a dimensional strengthening of known integral estimates for BMO on balls. It relies on a further transference, from $\BMO_{\scriptscriptstyle H}(\rn)$ to $\BMO_*(\rn).$ Specifically, we show that $\|\varphi\|_{\scriptscriptstyle H}\le C\sqrt n\|\varphi\|_*,$ which means that an estimate on $\BMO_*(I)$ automatically produces a ``$\sqrt n$''-estimate for $\BMO_*(\rn).$ We have the following schematic:
$$
\BMO_*(I)\xRightarrow{\text{dimension-free~}} \BMO_{\scriptscriptstyle H}(\rn) \xRightarrow{\sqrt n~} \BMO_*(\rn).
$$
In particular, we show that the John--Nirenberg constant of $\BMO_*(\rn)$ decays no faster than $\frac1{\sqrt n}.$ This is a notable improvement from what is currently available, as the usual methods for proving estimates for $\BMO$ on balls involve dyadic decompositions, which produce exponential dependence on dimension (thus, exponential decay for the John--Nirenberg constant).

We have chosen here to focus on two classical semigroup kernels, $P_t$ and $H_t,$ and two of the best known averaging classes, $\BMO$ and $A_p$. The proofs of our main theorems given in Section~\ref{proofs} depend on the probabilistic representation of the kernel $K_t,$ which is particularly simple when $K=P$ or $K=H.$ However, the argument would also work for a more general Markovian semigroup, with appropriate adjustments to the probabilistic formalism. Likewise, our results also hold for much more general averaging classes -- specifically, the classes $A_\Omega$ defined in~\cite{sz}. However, unlike a general $A_\Omega$ domain, the Bellman domains for $\BMO$ and $A_p$ possess homogeneity (additive for BMO, multiplicative for $A_p$), and that allows for simple mollification procedures; see Section~\ref{proofs}. Absent such homogeneity, the mollification would be quite a bit more involved. Going further still, our arguments will work for averaging classes on domains in $\rn$ or even in general metric spaces, as long as one can define, say, the heat kernel. Of course, in such settings one would not have explicit formulas for the kernels such as~\eqref{kernels}, making it harder to express the estimates obtained through the classical norms or characteristics, which is something we do in Section~\ref{appl} below. We intend to consider general semigroups $K,$ general classes $A_\Omega,$ and, possibly, general domains elsewhere.

The rest of the paper is organized as follows. Section~\ref{bellman_defs} contains the necessary Bellman definitions. In Section~\ref{main}, we state the main inequalities for Bellman functions, Theorems~\ref{main_bmo} and~\ref{main_ap}, and their implications for integral estimates, Corollaries~\ref{cor_bmo}~and~\ref{cor_ap}. In Section~\ref{appl}, we obtain general estimates for $\BMO_*(\rn)$ (Corollary~\ref{cor_rn}); Theorem~\ref{thjn} then gives the new bound on the John--Nirenberg constant of $\BMO_*(\rn).$ Finally, in Section~\ref{proofs}, we prove the theorems from Section~\ref{main}.

\section{Bellman function definitions}
\label{bellman_defs}
In what follows, $f$ is a measurable non-negative function on $\mathbb{R};$ the numbers $\mu>0,$ $\delta>1,$ and $p>1$ are fixed; $I$ is a finite interval; and  $z=(y,t)$ is a point in $\rn\times\mathbb{R}_+.$ 

We first define Bellman functions for BMO:
\eq[bbmou]{
\bel{B}_*(x; \mu,f)=\sup\big\{\av{f\circ\eta}I:~\|\eta\|^{}_{*,I}\le\mu,~\av{\eta}I=x_1,~\av{\eta^2}I=x_2\big\},
}
\eq[bbmok]{
\bel{B}_{\scriptscriptstyle K}(x; \mu,f,n)=\sup\big\{(f\circ\varphi)(z):~\|\varphi\|_{\scriptscriptstyle K}\le\mu,~\varphi(z)=x_1,~\varphi^2(z)=x_2\big\}.
}
An easy rescaling argument shows that $\bel{B}_*$ does not depend on $I$ and $\bel{B}_{\scriptscriptstyle K}$ does not depend on $z.$ Both $\bel{B}_*$ and $\bel{B}_{\scriptscriptstyle K}$ are defined, as functions of $x,$ on the parabolic domain
\eq[ombmo]{
\Omega_\mu:=\{x=(x_1,x_2)\in\mathbb{R}^2\colon~x_1^2\le x_2 \le x_1^2+\mu^2\}.
}
They also satisfy the following boundary condition:
$$
\bel{B}_*(x_1,x_1^2; \mu, f)=\bel{B}_{\scriptscriptstyle K}(x_1,x_1^2; \mu, f,n)=f(x_1),\quad x_1\in\mathbb{R}.
$$
This is because the only functions $\eta$ on $I$ such that $\av{\eta}I^2=\av{\eta^2}I$ are constants; the same is true for functions $\varphi$ on $\rn$ such that $\varphi(z)^2=\varphi^2(z).$
 
The analogs of definitions~\eqref{bbmou} and~\eqref{bbmok} for $A_p,$ $1<p<\infty,$ are as follows:
\eq[bapu]{
\bel{D}_p(x; \delta,f)=\sup\big\{\av{f\circ v}I:~[v]_{p,I}\le\delta,~\av{v}I=x_1,~\av{v^{-\frac1{p-1}}}I=x_2\big\},
}
\eq[bapk]{
\bel{D}_{p,{\scriptscriptstyle K}}(x; \delta,f,n)=\sup\big\{(f\circ w)(z):~[w]_{p}^{\scriptscriptstyle K}\le\delta,~w(z)=x_1,~w^{-\frac1{p-1}}(z)=x_2\big\}.
}
For $p=\infty,$ we define
\eq[bainfu]{
\bel{D}_\infty(x; \delta,f)=\sup\big\{\av{f\circ v}I:~[v]_{\infty,I}\le\delta,~\av{v}I=x_1,~\av{\log v}I=x_2\big\},
}
\eq[bainfk]{
\bel{D}_{\infty,{\scriptscriptstyle K}}(x; \delta,f,n)=\sup\big\{(f\circ w)(z):~[w]_{\infty}^{\scriptscriptstyle K}\le\delta,~w(z)=x_1,~\log w(z)=x_2\big\}.
}
Again, we see that $\bel{D}_p$ and $\bel{D}_{p,{\scriptscriptstyle K}}$ do not depend on $I$ and $z,$ respectively. For $p<\infty,$ the domain of definition for $\bel{D}_{p}$ and $\bel{D}_{p,{\scriptscriptstyle K}}$ is
\eq[omap]{
\Omega_{p,\delta}:=\{x=(x_1,x_2)\in\mathbb{R}^2\colon~x_1>0,~x_2>0,~1\le x_1x_2^{p-1}\le \delta\},
}
and the natural boundary condition is
$$
\bel{D}_p\big(x_1, x_1^{-1/(p-1)}; \delta, f\big)=\bel{D}_{p,{\scriptscriptstyle K}}\big(x_1, x_1^{-1/(p-1)}; \delta, f,n\big)=f(x_1),\quad x_1>0.
$$
For $p=\infty,$ the domain is
\eq[omainf]{
\Omega_{\infty,\delta}:=\{x=(x_1,x_2)\in\mathbb{R}^2\colon~x_1>0,~1\le x_1e^{-x_2}\le \delta\},
}
and the boundary condition is
$$
\bel{D}_\infty\big(x_1, \log x_1; \delta, f\big)=\bel{D}_{\infty,{\scriptscriptstyle K}}\big(x_1, \log x_1; \delta, f,n\big)=f(x_1),\quad x_1>0.
$$

The study of Bellman functions of the form~\eqref{bbmou} for BMO on an interval started with~\cite{sv}, where the first such function was computed for $f(t)=e^t.$ 
It was continued in~\cite{sv1}, where the case $f(t)=|t|^p,~p>0,$ was dealt with and the beginnings of a general PDE- and geometry-based theory for a general $f$ were laid out. That theory was fully developed in~\cite{long1} and~\cite{long2}. As a result, one can now compute $\bel{B}_*(\,\cdot\,; \mu, f)$ for any $f$ satisfying mild regularity conditions. Moreover, the techniques developed in these papers for BMO have been extended in~\cite{short} to other averaging classes, such as $A_p;$ thus, one can now compute the function $\bel{D}_{p}$ under similar assumptions on $f.$ However, nothing has been known about the functions $\bel{B}_{\scriptscriptstyle K}$ and $\bel{D}_{p,{\scriptscriptstyle K}}.$

\section{The main results}
\label{main}
Here are our main theorems connecting the Bellman functions for the classical $\BMO$ and $A_p$ with their $K$-analogs.
\begin{theorem}
\label{main_bmo}
For any non-negative measurable function $f$ on $\mathbb{R}$ and any numbers $\mu$ and $\tilde\mu,$ such that  $0<\tilde\mu<\mu,$ we have
\eq[bmo_case]{
\bel{B}_{\scriptscriptstyle K}(x; \tilde\mu,f,n)\le \bel{B}_*(x; \mu,f),~x\in\Omega_{\tilde\mu}.
}
\end{theorem}

\begin{theorem}
\label{main_ap}
For any non-negative measurable function $f$ on $(0,\infty)$ and any numbers
$\delta, \tilde\delta, p$ such that $1<\tilde\delta<\delta$ and $1<p\le\infty,$ we have 
\eq[ap_case]{
\bel{D}_{p,{\scriptscriptstyle K}}(x; \tilde\delta,f,n)\le  \bel{D}_p(x; \delta,f),~x\in\Omega_{p,\tilde\delta}.
}
\end{theorem}
The proofs of these theorems are given in Section~\ref{proofs}. Their practical importance is captured by the following immediate corollaries, of which we prove the first one; the proof of the second one is completely analogous.
\begin{corollary}
\label{cor_bmo}
If for some function $C_f\colon [0,\infty)\to[0,\infty]$ the estimate
\eq[es1]{
\av{f(\eta-\av{\eta}I)}I\le C_{f}(\mu)
}
holds for any interval $I$ and any $\eta\in \BMO(I)$ with $\|\eta\|^{}_{*,I}\le\mu,$ then the estimate
\eq[es2]{
f\big(\varphi-\varphi(z)\big)(z)\le C_{f}(\mu)
}
holds for all $\varphi\in\BMO(\rn)$ with $\|\varphi\|_{\scriptscriptstyle K}<\mu$ and all $z\in\mathbb{R}^{n+1}_+.$ Thus, all estimates~\eqref{es2} are dimension-free.
\end{corollary}
\begin{proof}
Inequality $\eqref{es1}$ is equivalent to the inequality $\bel{B}_*(0,x_2;\,\mu, f)\le C_f(\mu)$ for any $0\le x_2\le\mu^2,$ thus, by~\eqref{bmo_case} we have
\begin{align*}
f\big(\varphi-\varphi(z)\big)(z)&\le \bel{B}_{\scriptscriptstyle K}\big(0, \varphi^2(z)-\varphi(z)^2;\,\|\varphi\|_{\scriptscriptstyle K}, f, n\big)\\
&\le \bel{B}_{*}\big(0, \varphi^2(z)-\varphi(z)^2;\,\mu, f\big)\le C_f(\mu).\qedhere
\end{align*}
\end{proof}
\begin{corollary}
\label{cor_ap}
If for some function $E_f\colon [1,\infty)\to[0,\infty]$ the estimate 
\eq[es3]{
\Big\langle f\Big(\frac v{\av{v}I}\Big)\Big\rangle_I \le E_{f}(\delta)
}
holds for any interval $I$ and any $v\in A_p(I)$ with $[v]_{p,I}\le\delta,$ then the estimate
\eq[es4]{
f\Big(\frac w{w(z)}\Big)(z) \le E_{f}(\delta)
}
holds for all $w\in A_p^K(\rn)$ with $[w]_{p}^{\scriptscriptstyle K}<\delta$ and all $z\in\mathbb{R}^{n+1}_+.$ Thus, all such estimates are dimension-free.
\end{corollary}
\section{Estimates for $\BMO_*(\rn)$ and the John--Nirenberg constant}
\label{appl}
In this section, we will use $\lesssim$ and $\gtrsim$ in inequalities that hold up to an absolute, dimension-free multiplicative constant. We first establish an explicit dimensional bound on $\|\varphi\|_{_H}$ in terms of $\|\varphi\|_*$ and a pair of simple inequalities relating heat averages to Lebesgue averages over balls. The general result for $\BMO_*(\rn)$ is Corollary~\ref{cor_rn}, while a new dimensional bound for the John--Nirenberg constant of $\BMO_*(\rn)$ is given in~Theorem~\ref{thjn}. A note about notation: in Propositions~\ref{prop} and~\ref{prop1}, as well as in Corollary~\ref{cor_rn}, all extensions of the form $\varphi(z)$ are heat extensions. In the rest of the section, $K$ can be taken to be either $H$ or $P,$ unless expressly specified.
\begin{prop}
\label{prop}
If $\varphi\in\BMO(\rn),$ then
$$
\|\varphi\|_{_H}\lesssim \sqrt n\,\|\varphi\|_*.
$$
\end{prop}
\begin{proof}
Let $z_0=(0,\frac 14),$ where $0$ is the origin in $\rn.$ Due to scale invariance, it suffices to prove the inequality
$$
\Delta:=\varphi^2(z_0)-\varphi(z_0)^2\lesssim n \,\|\varphi\|^2_*.
$$
Let $h(x)=H_{1/4}(x)=\frac1{\pi^{n/2}}\,e^{-|x|^2}.$ Then
$$
\Delta=\int_{\rn}\int_{\rn} h(x)h(y)\big(\varphi(x)-\varphi(y)\big)^2\,dx\,dy.
$$
Integration by parts gives
\begin{align*}
\Delta&=\frac4{\pi^n}\,\int_0^\infty\int_0^\infty r_1r_2e^{-r_1^2-r_2^2}\Big[\int_{B_{r_1}}\int_{B_{r_2}}\big(\varphi(x)-\varphi(y)\big)^2\,dx\,dy\Big]\,dr_1\,dr_2,
\end{align*}
where $B_r$ denotes the ball of radius $r$ centered at $0.$ Now,
\begin{align*}
\int_{B_{r_1}}\int_{B_{r_2}}\big(\varphi(x)-\varphi(y)\big)^2\,dx\,dy&=
|B_{r_1}||B_{r_2}|\big(\av{\varphi^2}{B_{r_1}}+\av{\varphi^2}{B_{r_2}}-2\av{\varphi}{B_{r_1}}\av{\varphi}{B_{r_2}}\big)\\
&\le(V_n)^2r_1^nr_2^n\big(2\|\varphi\|^2_*+(\av{\varphi}{B_{r_1}}-\av{\varphi}{B_{r_2}})^2\big)\\
&\lesssim \|\varphi\|^2_*(V_n)^2r_1^nr_2^n \left(1+n^2\log^2\Big(\frac{r_1}{r_2}\Big)\right),
\end{align*}
where $V_n:=\frac{\pi^{n/2}}{\Gamma(\frac n2+1)}$ is the volume of the unit ball in $\rn$ and on the last step we used the elementary estimate 
$$
|\av{\varphi}{B_{r_1}}-\av{\varphi}{B_{r_2}}|\lesssim n\Big|\log\Big(\frac{r_1}{r_2}\Big)\Big|\,\|\varphi\|_*.
$$
Putting everything together, we have
$$
\Delta\lesssim\frac{n^2\|\varphi\|^2_*}{\big(\Gamma(\frac n2+1)\big)^2}\,\int_0^\infty\int_0^\infty r_1^{n+1}r_2^{n+1}e^{-r_1^2-r_2^2}\log^2\Big(\frac{r_1}{r_2}\Big)\,dr_1\,dr_2.
$$
The last integral can be seen to equal $\frac18\,(\Gamma(\frac n2+1))^2\,\Psi_1(\frac n2+1),$ where $\Psi_1(z):=\frac{d^2}{dz^2}\log\Gamma(z)$ is the first polygamma function. Since $\Psi_1(\frac n2+1)=O(\frac1n)$ as $n\to\infty,$ the proof is complete.
\end{proof}

\begin{prop}
\label{prop1}
If $B=B(x,r)$ is the ball with radius $r$ centered at a point $x\in\rn,$ then there exists a point $z_B\in\mathbb{R}_{n+1}^+$ such that for any non-negative function $g$ on $\rn$ we have 
\eq[pr1]{
\av{g}B\lesssim \sqrt n\,g(z_B).
}
Furthermore, if $\varphi\in\BMO(\rn),$ then
\eq[pr2]{
|\varphi(z_B)-\av{\varphi}B|\lesssim (\log n+1)\,\|\varphi\|_{\scriptscriptstyle H}.
}
\end{prop}
\begin{proof}
Let $t_B=\frac{r^2}{2n}$ and $z_B=(x,t_B).$ Then, for $y\in B,$ 
$$
H_{t_B}(x-y)\ge \frac1{r^n}\,\Big(\frac n{2\pi e}\Big)^{n/2} =\frac{V_n}{|B|}\,\Big(\frac n{2\pi e}\Big)^{n/2}
=\frac1{|B|}\,\frac1{\Gamma(\frac n2+1)}\Big(\frac n{2e}\Big)^{n/2}\gtrsim \frac1{\sqrt n}\,\frac1{|B|},
$$
where the last inequality follows from Stirling's formula. This proves~\eqref{pr1}.

The proof of~\eqref{pr2} is more interesting. It was shown in~\cite{sv1} (cf. Th. 2.5 of that paper) that for any interval $I$ and any function $\eta\in \BMO(I)$ with $\|\eta\|_{*,I}<1$ one has
$$
\av{e^{|\eta-\av{\eta}{\scriptscriptstyle I}|}}I\le \frac1{1-\|\eta\|_{*,I}}.
$$
By Corollary~\ref{cor_bmo} with $f(s)=e^{|s|}$ and $C_f(s)=\frac1{1-s},$ if $\psi\in\BMO(\rn)$ and $\|\psi\|_{\scriptscriptstyle H}=\frac12,$ then
\eq[pr3]{
e^{|\psi-\psi(z_{\scriptscriptstyle B})|}(z_B)\le \frac1{1-\|\psi\|_{\scriptscriptstyle H}}=2.
}
Hence, for some absolute constant $C,$
\begin{align*}
|\psi(z_B)-\av{\psi}B|&\le \av{|\psi-\psi(z_B)|}B\le \log\Big(\av{e^{|\psi-\psi(z_{\scriptscriptstyle B})|}}B\Big) \\
&\le C+\log\Big(\sqrt n\,e^{|\psi-\psi(z_{\scriptscriptstyle B})|}(z_B)\Big)
\le C+\log\big(2\sqrt n\big)\lesssim \log n+1,
\end{align*}
where we first used the triangle inequality, then Jensen's inequality, then~\eqref{pr1}, and, finally,~\eqref{pr3}. Replacing $\psi$ with $\frac12\,\frac{\varphi}{\|\varphi\|_{\scriptscriptstyle H}},$ we obtain~\eqref{pr2}
\end{proof}

We now give two general inequalities for integral functionals on $\BMO_*(\rn)$ in the spirit of~Corollary~\ref{cor_bmo}. The first one is more transparent, but it involves the difference $\varphi-\varphi(z_B)$ instead of the usual $\varphi-\av{\varphi}B.$ The second inequality does give estimates in terms of $\varphi-\av{\varphi}B,$ but it requires partial knowledge of the one-dimensional Bellman function~$\bel{B}_*$ defined by~\eqref{bbmou}. Fortunately, such functions can now be computed for any functional $\av{f\circ\eta}I$ under mild regularity assumptions on $f;$ see~\cite{long1,long2}.

\begin{corollary}
\label{cor_rn}
If for some function $C_f\colon [0,\infty)\to[0,\infty]$ the estimate
\eq[es11]{
\av{f(\eta-\av{\eta}I)}I \le C_{f}(\mu)
}
holds for any interval $I$ and any $\eta\in \BMO(I)$ with $\|\eta\|^{}_{*,I}\le\mu,$ then the estimate
\eq[es21]{
\av{f(\varphi-\varphi(z_B))}B\lesssim \sqrt n\,C_{f}(\mu)
}
holds for all $\varphi\in\BMO(\rn)$ with $\|\varphi\|_*\le c\frac{\mu}{\sqrt n},$ where $c$ is an absolute constant, and all balls $B,$ with $z_B$ given by Proposition~\ref{prop1}.
Furthermore, for such $\varphi,$ 
\eq[es31]{
\av{f(\varphi-\av{\varphi}B)}B\lesssim \sqrt n \sup_{x\in \omega_\mu}\bel{B}_*(x; \mu,f),
}
where the Bellman function $\bel{B_*}$ is defined by~\eqref{bbmou} and
$$
\omega_\mu:=\{(x_1,x_2)\colon~|x_1|\le (\log n+1)\mu,~x_1^2\le x_2\le x_1^2+\mu^2\}.
$$
\end{corollary}
\begin{proof}
If $c$ is chosen so that $ \|\varphi\|_{\scriptscriptstyle H}<\mu$ (such a $c$ exists by Proposition~\ref{prop}), then~\eqref{es21} is immediate from~\eqref{es2} with $K=H$ and~\eqref{pr1} with $g=f\circ\varphi.$ 

To prove~\eqref{es31}, observe that 
\begin{align*}
\av{f(\varphi-\av{\varphi}B)}B&\lesssim \sqrt n f(\varphi-\av{\varphi}B)(z_B)\le
\sqrt n\,
\bel{B}_{\scriptscriptstyle H}\big(x; \|\varphi\|_{\scriptscriptstyle H},f,n\big)\le\sqrt n\,\bel{B}_*\big(x; \mu,f\big),
\end{align*}
where $x=(x_1,x_2):=(\varphi(z_B)-\av{\varphi}B,\,\varphi^2(z_B)-2\varphi(z_B)\av{\varphi}B+\av{\varphi}B^2).$ By Proposition~\ref{prop1}, $|x_1|\le \tilde{c} (\log n+1)\|\varphi\|_{\scriptscriptstyle H}$ for some constant $\tilde{c}.$ By adjusting $c$ in the assumption  $\|\varphi\|_*\le c\frac{\mu}{\sqrt n}$ we can ensure that $\tilde{c}\|\varphi\|_{\scriptscriptstyle H}\le\mu$ and, thus, that $|x_1|\le (\log n+1)\mu.$ In addition, 
$$
x_1^2\le x_2=x_1^2+\varphi^2(z_B)-\varphi(z_B)^2\le x_1^2+\|\varphi\|_{\scriptscriptstyle H}^2<x_1^2+\mu^2.
$$ 
Therefore, $x\in\omega_\mu$ and~\eqref{es31} follows.
\end{proof}
We come to the main result of this section. The John--Nirenberg inequality~\cite{jn} says that there exist constants $\ve_*>0$ and $C_*>0$ such that for any $\varphi\in\BMO(\rn),$ any ball $B\subset\rn,$ and any number $\lambda\ge0,$
\eq[jn*]{
\frac1{|B|}|\,\{t\in B\colon~|\varphi(t)-\av{\varphi}J|>\lambda\}|\le C_* e^{-\frac{\ve_*\lambda}{\|\varphi\|_*}}.
}
The analog of \eqref{jn*} for $\BMO_{\scriptscriptstyle K}$ is this: there exist constants $\ve_{\scriptscriptstyle K}>0$ and $C_{\scriptscriptstyle K}>0$ such that
\eq[jnK]{
\chi^{}_{\scriptstyle\{|\varphi-\varphi(z)|>\lambda\}}(z)\le C_{\scriptscriptstyle K} e^{-\frac{\ve^{}_{\scriptscriptstyle K}\lambda}{\|\varphi\|_K}},~\text{for all~}z\in\mathbb{R}^{n+1}_+.
}

All constants above depend on dimension. Let
$$
\ve^{\scriptscriptstyle\rm JN}_*(n)=\sup\{\ve_*>0\colon~\exists C_*>0\colon~\eqref{jn*}\text{~holds}\},
$$
$$
\ve^{\scriptscriptstyle\rm JN}_{\scriptscriptstyle K}(n)=\sup\{\ve_{\scriptscriptstyle K}>0\colon~\exists C_{\scriptscriptstyle K}>0\colon~\eqref{jnK}\text{~holds}\}.
$$
We call $\ve^{\scriptscriptstyle\rm JN}_*$ and $\ve^{\scriptscriptstyle\rm JN}_{\scriptscriptstyle K}$ the John--Nirenberg constant of $\BMO_*$ and $\BMO_{\scriptscriptstyle K},$ respectively. This constant can be similarly defined for any other choice of $\BMO$ norm, including BMO on all cubes. In all cases, the size of the constant -- and, specifically, its dimensional behavior -- are of interest; see~\cite{c, cys}. The classical proof of ~\eqref{jn*} yields exponential decay of $\ve^{\scriptscriptstyle\rm JN}_*(n)$ in $n;$ see~\cite{jn}. In~\cite{wik}, Wik showed that the analog of $\ve^{\scriptscriptstyle\rm JN}_*(n)$ for BMO on cubes decays no faster than $n^{-1/2}.$ His beautiful proof relied heavily on the product structure of the cube. To our knowledge, until now no results better than exponential have been known for $\BMO_*$ or $\BMO_{\scriptscriptstyle K}.$
\begin{theorem}
\label{thjn}
\eq[inK]{
\ve_{\scriptscriptstyle K}^{\scriptscriptstyle\rm JN}(n)\ge1.
}
Consequently,
\eq[in*]{
\ve_*^{\scriptscriptstyle\rm JN}(n)\gtrsim n^{-1/2}.
}
\end{theorem}
\begin{proof}
It was shown in \cite{vv} that for an interval $I$ and $\eta\in\BMO(I),$ 
$$
\frac1{|I|}|\,\{t\in I\colon~|\eta(t)-\av{\eta}I|>\lambda\}|\le e^{1-\frac{\lambda}{\|\eta\|_{*,I}}},
$$
which is statement~\eqref{es1} with $f(s)=\chi^{}_{\{|s|
>\lambda\}}(s)$ and $C_f(s)=e^{1-\frac\lambda s}.$ Therefore, by Corollary~\ref{cor_bmo},~\eqref{jnK} holds with $C_{\scriptscriptstyle K}=e$ and $\ve_{\scriptscriptstyle K}=1.$ This proves~\eqref{inK}.

Now, take $K=H.$ Fix a ball $B\subset \rn$ and let $z_B$ be given by Proposition~\ref{prop1}. Then for any $\varphi\in\BMO(\rn),$ 
$
|\varphi(z_B)-\av{\varphi}B|\lesssim (\log n+1)\, \|\varphi\|_{\scriptscriptstyle H}.
$
Thus, for some dimensional constant $c(n)$ and for $\lambda\ge c(n)\|\varphi\|_{\scriptscriptstyle H},$
\begin{align*}
\frac1{|B|}\,|\{t\in B\colon~|\varphi(t)-\av{\varphi}B|>\lambda\}|&\le 
\frac1{|B|}\,|\{t\in B\colon~|\varphi(t)-\varphi(z_B)|>\lambda-c(n)\|\varphi\|_{\scriptscriptstyle H}\}|\\
&\lesssim \sqrt n\,\chi^{}_{\scriptstyle\{|\varphi-\varphi(z_B)|>\lambda-c(n)\|\varphi\|_H\}}(z_B)\\
&\lesssim \sqrt n\, e^{-\frac{\lambda-c(n)\|\varphi\|_H}{\|\varphi\|_H}}=C(n)\,e^{-\frac{\lambda}{\|\varphi\|_H}},
\end{align*}
where $C(n)$ is another dimensional constant. By adjusting $C(n)$ we can assume that $\lambda\ge0.$ By Proposition~\ref{prop}, $\|\varphi\|_{\scriptscriptstyle H}\lesssim \sqrt n\|\varphi\|_*,$ and the proof is thus complete.
\end{proof}
\begin{remark}
The constant $C(n)$ in the proof above can be optimized, but it is not important for our purposes.
We do not know if the constant $\sqrt n$ in Proposition~\ref{prop} is sharp, though we suspect that it is. The analogous statement for the Poisson kernel is  $\|\varphi\|_{\scriptscriptstyle P}\lesssim n\|\varphi\|_*,$ thus using $K=P$ in the proof would yield the estimate $\ve^{\scriptscriptstyle\rm JN}_*(n)\gtrsim n^{-1},$ which is worse than~\eqref{in*}. 
\end{remark}

\section{Proofs of the main theorems}
\label{proofs}
Here we first prove Theorem~\ref{main_bmo}, and then describe what changes are necessary in the proof of Theorem~\ref{main_ap}, which is largely the same.
The key ingredient in the proof is the following result, which is a special case of a general theorem from~\cite{sz} (cf. the theorem on p.~230 of that paper).
\begin{theorem}[\cite{sz}]
\label{sz}
The function $\bel{B}_*(\,\cdot\,;\, \mu,f)$ is the minimal locally concave function $U$ on $\Omega_\mu$ satisfying the boundary condition $U(x_1,x_1^2)=f(x_1),$ $x_1 \in \mathbb{R}.$
\end{theorem}
By ``locally concave'' we mean a function that is concave on any convex subset of $\Omega_\mu.$ Theorem~\ref{main_bmo} is an immediate corollary of Theorem~\ref{sz} and the following lemma.
\begin{lemma}
\label{ito}
Let $U$ be a non-negative, locally concave function on $\Omega_\mu,$ $\varphi\in\BMO(\rn)$ with $\|\varphi\|_{\scriptscriptstyle K}<\mu,$ and $z_0\in\mathbb{R}^{n+1}_+.$ Then
\eq[uge]{
U(\varphi(z_0),\varphi^2(z_0))\ge U(\varphi,\varphi^2)(z_0).
}
\end{lemma}
\begin{proof}[Proof of Theorem~\ref{main_bmo}]

By Theorem~\ref{sz}, $\bel{B}_*(\,\cdot\,;\,\mu,f)$ can be used as $U$ in Lemma~\ref{ito}. Fix a point $z_0\in\mathbb{R}^{n+1}_+$ and take any point $x\in\Omega_{\tilde\mu}.$ It is easy to show that the set
$$
F_{x,\tilde\mu,z_0}:=\{\varphi\in\BMO(\rn)\colon~\|\varphi\|_{\scriptscriptstyle K}\le\tilde\mu; (\varphi(z_0),\varphi^2(z_0))=x\}
$$ 
is non-empty. Now, take supremum of the right-hand side of~\eqref{uge} over all elements of $F_{x,\tilde\mu,z_0}.$ That supremum is precisely $\bel{B}_{\scriptscriptstyle K}(x;\,\tilde\mu,f,n).$
\end{proof}

\begin{proof}[Proof of Lemma~\ref{ito}] We first prove the lemma under an additional assumption that $U$ is $C^2,$ then for continuous $U,$ and, finally, for general $U.$

Let $\tilde\mu=\|\varphi\|_{\scriptscriptstyle K}.$ First, assume that $U$ is non-negative and $C^2$ in a neighborhood of $\Omega_{\tilde\mu}$ and that its Hessian ${\rm d}^2 U$ is non-positive definite. Let
$$
L_+=\{(y,t)\in\mathbb{R}^{n+1}\colon~y\in\rn, t>0\},\quad  L_0=\{(y,0)\in\mathbb{R}^{n+1}\colon~y\in\rn\}.
$$
We will use the following probabilistic representation of the kernel $K_t$ (see, e.g. \cite{iw}): there is an It\^o process $Z_t$ starting at $z_0=(y_0,t_0),$ arriving almost surely at $L_0$ in finite time, and such that for any non-negative function $g$ on $\rn,$
\eq[gz]{
g(z_0)=\mathbb{E}\,g(Z_\infty).
}
For $K=P,$ $Z_t=Z_t^P$ is the $(n+1)$-dimensional Brownian motion starting at $z_0$ and stopped at $L_0$ (when its last component is 0): $Z_t^P=(x_0+B^n_{\min(t,\tau)}, t_0+B^1_{\min(t,\tau)}),$ where $\tau=\min\{s\colon B^1_s+t_0=0\}$ is the stopping time.
For $K=H,$ $Z_t = Z_t^H$ is given by the $n$-dimensional Brownian motion starting at $y_0$ for the first $n$ components and the variable $t_0-\frac12 t$ for the last component, also stopped at $L_0$ (at the non-random time $t=2t_0$): $Z_t^H=(x_0+B^n_{\min(t,2t_0)},t_0-\frac12 \min(t,2t_0)).$

Now,
$\Phi_t:= (\varphi(Z_t), \varphi^2(Z_t))$ is an It\^o martingale:
$$
d\Phi_t = \sigma_t\, dB^2_t
$$
for an appropriate diffusion matrix $\sigma_t$. In addition, since $\|\varphi\|_{\scriptscriptstyle K}=\tilde\mu,$ $\Phi_t$ takes values in $\Omega_{\tilde\mu}.$ 
It\^o's formula and the non-positivity of ${\rm d}^2U$ imply that the expectation
$
\mathbb{E}\,U\big(\Phi_t\big)
$
is non-increasing in~$t$: 
$$
\frac{d}{dt}\,\mathbb{E}\,U(\Phi_t) = \frac{1}{2}\,\mathbb{E}\,\mathrm{Tr}\,[\sigma_t^T {\rm d}^2U(\Phi_t) \sigma_t]\leq 0.
$$
Therefore,
\begin{align*}
U(\varphi(z_0),\varphi^2(z_0))&=\mathbb{E}\,U\big(\Phi_0\big)
\ge \mathbb{E}\,U\big(\Phi_\infty\big)
= U(\varphi,\varphi^2)(z_0),
\end{align*}
where we have used~\eqref{gz} and the fact that $Z_\infty\in L_0$ a.s. and, thus, $\varphi^2(Z_\infty)=\varphi(Z_\infty)^2.$ 

We have proved the lemma for smooth $U.$ Now, suppose that $U$ is continuous on $\Omega_\mu.$ We construct a sequence of non-negative, smooth, locally concave functions $U_j$ on a neighborhood of $\Omega_{\tilde\mu}$ that converges to $U$ pointwise on $\Omega_{\tilde\mu}.$ For each such function $U_j$ inequality~\eqref{uge} is already proved, and one can take the limit on the right using Fatou's lemma. 

In order to construct $U_j,$ we employ a convolution-like mollifier using the additive homogeneity of the domain. Fix some parameters $\rho_j>0$ and $r_j>0$ to be chosen later. Take a function $\psi_j$ in $C^\infty(\mathbb{R}^2)$ with support in the ball of radius $r_j$ centered at the origin and satisfying $\psi_j \geq 0$ and $\int\psi_j=1.$ We define 
\eq[eq3]{
U_j(x_1,x_2) = \iint_{\mathbb{R}^2} U(x_1-t_1, x_2+\rho_j- 2x_1t_1+t_1^2-t_2) \psi_j(t_1,t_2)\, dt_1\, dt_2.
}
This function is obviously smooth; it is also locally concave because for each fixed $(t_1,t_2)$ the function 
$$
(x_1,x_2)\mapsto U(x_1-t_1, x_2+\rho_j- 2x_1t_1+t_1^2-t_2)
$$ 
is locally concave. $U_j$ is correctly defined whenever
\eq[eq2]{
r_j-\rho_j \leq x_2-x_1^2 \leq \mu^2-r_j- \rho_j,
}
because in this case we have $(x_1-t_1, x_2+\rho_j- 2x_1t_1+t_1^2-t_2) \in \Omega_\mu$ for $(t_1,t_2) \in \mathrm{supp}(\psi_j)$. We may choose the sequences $\{r_j\}$ and $\{\rho_j\}$ so that $r_j\to0,$ $\rho_j\to0,$ $r_j-\rho_j<0,$ and  $r_j+ \rho_j<\mu^2-\tilde\mu^2.$ The function $U_j$ is now defined on some neighborhood of $\Omega_{\tilde\mu},$ and by the continuity of $U$ we have $U_j\to U$ on $\Omega_{\tilde\mu}.$ This completes the proof for the continuous case.

It remains to prove~\eqref{uge} for any locally concave function $U$ on $\Omega_\mu.$ If $\varphi^2(z_0)=\varphi(z_0)^2,$ \eqref{uge} holds with equality, so let us assume that $\varphi^2(z_0)>\varphi(z_0)^2.$ Define a new sequence of functions, $\{V_j\},$ by
\eq[vn]{
V_j(x_1,x_2)=U\big(x_1, x_2+\textstyle{\frac{\mu^2-\tilde\mu^2}j}\big).
}
Then each $V_j$ is defined on $\Omega_{\bar\mu}$ where $\bar\mu^2:=\frac{\mu^2+\tilde\mu^2}2.$ Furthermore, each $V_j$ is continuous, locally concave, and non-negative on $\Omega_{\bar\mu}.$ By the continuous case shown above, 
$$
V_j(\varphi(z_0),\varphi^2(z_0))\ge V_j(\varphi,\varphi^2)(z_0).
$$
The left-hand side converges to $U(\varphi(z_0),\varphi^2(z_0))$ as $j\to\infty$ since $(\varphi(z_0),\varphi^2(z_0))$ is an interior point of $\Omega_\mu,$ which means that $U$ is continuous at that point. As for the right-hand side, we have a pointwise inequality
$$
\liminf V_j(\varphi, \varphi^2)=\liminf U\big(\varphi, \varphi^2+\textstyle{\frac{\mu^2-\tilde\mu^2}j}\big)\ge U(\varphi,\varphi^2),
$$
since $U(x_1,\,\cdot\,)$ is concave on the interval $[x_1^2, x_1^2+\mu^2].$ An application of Fatou's lemma finishes the proof.
\end{proof}
The proof of Theorem~\ref{main_ap} is exactly the same as that of Theorem~\ref{main_bmo}, except Theorem~\ref{sz} and Lemma~\ref{ito} are replaced with the following two analogs.
\begin{theorem}[\cite{sz}]
\label{sz1}
~
For $1<p<\infty,$ the function $\bel{D}_p(\,\cdot\,; \delta,f)$ is the minimal locally concave function $U$ on $\Omega_{p,\delta}$ satisfying the boundary condition $U(x_1,x_1^{-1/(p-1)})=f(x_1),$ $x_1 >0$.

The function $\bel{D}_\infty(\,\cdot\,; \delta,f)$ is the minimal locally concave function $U$ on $\Omega_{\infty,\delta}$ satisfying the boundary condition $U(x_1,\log x_1)=f(x_1),$ $x_1 >0$.
\end{theorem}

\begin{lemma}
\label{ito1}
Let $U$ be a non-negative, locally concave function on $\Omega_{p,\delta},$ $w\in A_p^K(\rn)$ with $[w]_{p}^{\scriptscriptstyle K}<\delta,$ and $z_0\in\mathbb{R}^{n+1}_+.$ 

If $1<p<\infty,$ then
\eq[ugep]{
U(w(z_0),w^{-\frac1{p-1}}(z_0))\ge U(w,w^{-\frac1{p-1}})(z_0).
}
If $p=\infty,$ then
\eq[ugeinfty]{
U(w(z_0),\log w(z_0))\ge U(w,\log w)(z_0).
}
\end{lemma}
Theorem~\ref{sz1} is again a special case of the same general theorem from~\cite{sz}. To prove Lemma~\ref{ito1}, we let $\tilde\delta=[w]_{p}^{\scriptscriptstyle K}$ and
modify formulas~\eqref{eq3} and~\eqref{vn} as follows. For $1<p<\infty$ we let
$$
U_j(x_1,x_2)=\iint_{\mathbb{R}^2}U\big(e^{\rho_j-t_1} x_1, e^{-t_2} x_2\big)\psi_j(t_1,t_2)\,dt_1\,dt_2,
$$
where the non-negative sequences $\{r_j\}$ (the radius of the support of $\psi_j$) and $\{\rho_j\}$ are chosen so that $r_j\to0,$ $\rho_j\to0,$ $\rho_j>p r_j,$ and $\rho_j+p r_j< \log (\delta/\tilde\delta).$

Formula~\eqref{vn} is replaced with
$$
V_j(x_1,x_2)=V\big(\lambda_jx_1,\lambda_j^{\frac1{p-1}}x_2\big)
$$
for a sequence $\lambda_j\to1^+.$

For $p=\infty,$~\eqref{eq3} is replaced with
$$
U_j(x_1,x_2)=\iint_{\mathbb{R}^2}U\big(e^{\rho_j-t_1} x_1, x_2-t_2\big)\psi_j(t_1,t_2)\,dt_1\,dt_2
$$
with $\{r_j\}$ and $\{\rho_j\}$ chosen so that $r_j\to0,$ $\rho_j\to0,$ $\rho_j>2r_j,$ and $\rho_j+2r_j< \log (\delta/\tilde\delta).$ Instead of~\eqref{vn} we have
$$
V_j(x_1,x_2)=U(\lambda_j x_1, x_2-\log \lambda_j)
$$
for a sequence $\lambda_j\to1^+.$

\end{document}